\documentclass[11pt,reqno]{amsart}

\usepackage[utf8]{inputenc}

\usepackage{amsmath,amsthm,amssymb}
\usepackage{amssymb}

\usepackage{graphics}
\usepackage{hyperref}
\usepackage[usenames, dvipsnames]{xcolor}

\usepackage{mathtools}
\mathtoolsset{showonlyrefs}

\usepackage[square,sort,comma,numbers]{natbib}
\usepackage{verbatim}
\usepackage{array}
\usepackage{longtable}

\definecolor{darkblue}{rgb}{0.0,0.0,0.3}
\hypersetup{colorlinks,breaklinks,
  linkcolor=darkblue,urlcolor=darkblue,
anchorcolor=darkblue,citecolor=darkblue}

\theoremstyle{plain}
\newtheorem{theorem}{Theorem}[section]
\newtheorem*{theorem*}{Theorem}
\newtheorem{lemma}[theorem]{Lemma}
\newtheorem{proposition}[theorem]{Proposition}
\newtheorem*{proposition*}{Proposition}

\newtheorem*{corollary*}{Corollary}

\theoremstyle{definition}
\newtheorem{remark}[theorem]{Remark}

\numberwithin{equation}{section}

\title[Integer Sets of Large Harmonic Sum]{Integer Sets of Large Harmonic Sum which Avoid Long Arithmetic Progressions}
\author[Walker]{Alexander Walker \\ \today}

\begin{document}

\begin{abstract}
We give conditions under which certain digit-restricted integer sets avoid $k$-term arithmetic progressions. These sets and their harmonic sums can be computed efficiently. Through large-scale search, we identify integer sets avoiding arithmetic progressions of length $4$ and $10$ whose harmonic sums exceed earlier ``greedy'' constructions.
\end{abstract}

\maketitle

The Erd\H{o}s--Tur\'an conjecture on arithmetic progressions proposes that integer sets with divergent harmonic sums (so-called \emph{large} sets) must contain arithmetic progressions of arbitrary (finite) length. This conjecture is known to hold for integer sets of positive density (Szemeredi's theorem,~\cite{Szemeredi75}) and for the set of primes (the Green--Tao theorem,~\cite{GreenTao08}).

Let $\mathcal{S}_k$ denote the collection of sets of positive integers which avoid arithmetic progressions of length $k$.  Such sets will be called \emph{$k$-free} hereafter. The Erd\H{o}s--Tur\'an conjecture implies that the harmonic sum of any member of $\mathcal{S}_k$ is bounded. In fact, under Erd\H{o}s--Tur\'an, these harmonic sums would be \emph{uniformly} bounded as a function of $k$, due to a result of Gerver~\cite{Gerver77}.  Let
\[M_k := \sup_{T\in \mathcal{S}_k} \sum_{t \in T} 1/t,\]
which is finite for each $k$ if and only if the Erd\H{o}s--Tur\'an conjecture holds. Recent work of Bloom--Sisask~\cite{BloomSisask19} shows that $M_3$ is finite, but the finiteness of $M_k$ is otherwise unknown.

If finite, the growth rate of $M_k$ represents a refinement to the Erd\H{o}s--Tur\'an conjecture.  For this reason, lower bounds for $M_k$ appear in several works.  Berlekamp gave a construction in~\cite{Berlekamp68} which proved $M_k \geq \frac{1}{2} k \log 2$. This was improved in~\cite{Gerver77}, which established $M_k > (1-o(1))k \log k$.

Numerical lower bounds for $M_k$ for small $k$ have also received attention.  The current record for $M_3$ is held by Wr\'oblewski, who proved $M_3 \geq 3.00849$ by interlacing a greedy $3$-free set with a denser $3$-free packing due to Behrend~\cite{Behrend46, Wroblewski84}.

Let $G_k$ denote the lexicographically earliest $k$-free set.  The sets $G_k$ have reasonably large harmonic sums, especially when $k$ is prime and $G_k$ exhibits fractal self-similarity.  When $k$ is composite, the harmonic sums are less impressive. Heuristics from~\cite{GerverRamsey79} suggest that $G_4$ and $G_6$ have harmonic sums of $\approx 4.3$ and $\approx 6.9$, respectively.  Notably, the harmonic sum of $G_6$ is predicted to be \emph{less than} that of $G_5$, which is approximately $7.866$.

This article provides a new construction for infinite $k$-free sets. Fix an integer $b \geq 2$ and a proper subset of integers $S \subsetneq [0,b-1]$.  We define the \emph{Kempner set} $\mathcal{K}(S,b)$ as the set of non-negative integers whose base-$b$ digits are contained in $S$.  Kempner sets first appeared in~\cite{Kempner14}, and their arithmetic properties have been studied in~\cite{EMS99a}, \cite{EMS99b}, and~\cite{Maynard19}. The connection between Kempner sets and arithmetic progressions was first developed in~\cite{WalkerWalker20}.  In particular,~\cite{WalkerWalker20} proved that every Kempner set is $k$-free for some $k$.

Kempner sets are useful in the experimental study of $M_k$ because the lengths of their longest arithmetic progressions are easy to compute and their harmonic sums are computable in polynomial time (due to an algorithm of Baillie--Schmelzer in~\cite{BaillieSchmelzer08}).  Most importantly, they are also capable of producing large harmonic sums.  For example, the $4$-free set
\[\mathcal{K}(\{0,1,2,4,5,7\} ,11) +1 = \{1,2,3,5,6,8,12,13,14,16,17,19,23,24,\ldots\}\]
has harmonic sum $4.421746$.  This simple set exceeds the estimated harmonic sum of $G_4$ by a considerable margin and already sets a new lower bound for the supremum $M_4$.

We describe and implement algorithms which use Kempner sets to search for lower bounds for $M_k$.  Even with pruning, this search is time-consuming: the number of Kempner sets $\mathcal{K}(S,b)$ grows exponentially in $b$ and $b$ must be taken large before interesting results are found.
Our search is most successful in the case $k=4$, where our best result is the set
\[\mathcal{K}(\{0,1,2,4,5,9,10,11,14,16,17,18,21,24,30,37,39,41,42,45,47\},55) +1.\]
This set has harmonic sum $4.43975$ and sets a new lower bound for $M_4$. We also improve the lower bound for $M_{10}$: the set 
\begin{align*}
& \mathcal{K}(\{0, 1, 2, 3, 4, 5, 6, 7, 8, 10, 11, 12, 13, 14, 15, 17, 18, 19, 20, 21, \\
& \qquad 22, 24, 25, 26, 27, 28, 29, 31, 32, 33, 34, 35, 36, 38, 39, 40, 42, \\
& \qquad 43, 45, 46, 47, 48, 49, 50, 52, 53, 55, 56, 60, 61, 62, 68, 69, 71, 73\}, 77) + 1
\end{align*}
is $10$-free and has harmonic sum $14.056$, which improves the lower bound $M_{10} \geq 13.5905$ derived from $(G_7+3) \cup \{1,2,3\}$.

\section*{Acknowledgments}

This work was supported by the Additional Funding Programme for Mathematical Sciences, delivered by EPSRC (EP/V521917/1) and the Heilbronn Institute for Mathematical Research.

\section{Modular Arithmetic Progressions}

A set $S \subset [0,b-1]$ is called an \emph{arithmetic progression mod $b$} of length $k$ if there exists an arithmetic progression $A$ (in the ordinary sense) of length $k$ and common difference $\Delta$ for which $A \!\! \mod b$  lies in $S$ and $b \nmid \Delta$.  By extension, a set $S \subset [0,b-1]$ is called \emph{$k$-free mod $b$} if it contains no arithmetic progressions mod $b$ of length $k$.  Note that we do not require $\gcd(\Delta,b)=1$. This has some counter-intuitive implications; for example, it implies that $\{1,3,5\}$ is an arithmetic progression mod $6$ of infinite length.

One can test if a set $S$ is $k$-free mod $b$ by testing if the associated union of translates $\bigcup_{j=0}^{n-1} (S+jb)$ has no increasing $k$-term arithmetic progression with common difference less than $b$.  In the context of a depth-first search, one can create $k$-free sets mod $b$ by extending smaller $k$-free sets.

Alternatively, one can construct sets which are $k$-free mod $b$ by first fixing a finite $k$-free set and then specifying a $b$ which is sufficiently large. 

\begin{proposition} \label{prop:kfree_extension}
Let $S \subset [0,M]$ be a $k$-free set.  Then $S$ is $k$-free mod $b$ for all $b > 2M$.
\end{proposition}

\begin{proof} For the sake of contradiction, suppose that $S \subset [0,M]$ is $k$-free but that $S$ admits a $k$-term arithmetic progression mod $b$ for some $b> 2M$.

To be precise, suppose that $B = \{c+\Delta j: j \in [1,k]\}$ is equivalent mod $b$ to a subset of $S$ and that $0<\Delta <b$.  For $j \leq k$, let $q_j$ and $r_j$ be the quotient and remainder of $c+\Delta j$ upon division by $b$. Our assumption on $\Delta$ implies that $q_{j+1}$ equals $q_j$ or $q_j+1$.  If $q_{j+1} = q_{j}$ for all $j$, then $\{r_j\}$ is an arithmetic progression in $S$ of length $k$, a contradiction. Thus $q_{j+1} = q_j+1$ for some $j$, hence $\Delta \geq  b-M$ since $r_j \in [0,M]$.

On the other hand, if $q_{j+1} = q_j+1$ for all $j$, then $\{r_j\}$ is a (decreasing) arithmetic progression in $S$ of length $k$, again a contradiction.  Thus $q_{j+1} = q_j$ for some $j$, hence $\Delta \leq M$.  It follows that $b-M \leq \Delta \leq M$, which contradicts that $b> 2M$.
\end{proof}

Sets which are $k$-free mod $b$ can be used to produce $k$-free Kempner sets.  This is made precise in the following.

\begin{theorem} \label{prop:kfree_mod_p_to_kempner}
Fix $b \geq 3$.  If $S \subsetneq [0,b-1]$ is $k$-free mod $b$ and $0 \in S$, then $\mathcal{K}(S,b)$ is $k$-free.
\end{theorem}

\begin{proof} Suppose that $S$ is $k$-free mod $b$ and that $\mathcal{K}(S,b)$ contains the arithmetic progression $A = \{c+\Delta j \mid j \in [0,k-1]\}$.
If $b\nmid \Delta$, then $A$ and therefore $S$ contains residues in an arithmetic progression mod $b$. Thus $S$ contains a progression mod $b$ of length $k$ or of infinite length, so $S$ is not $k$-free.

Alternatively, suppose that $b \mid \Delta$ and let $c_0$ denote the base-$b$ units digit of $c$.  The arithmetic progression $(A-c_0)/b$ is contained in $\mathcal{K}(S,b)$ and has a smaller common difference.  We conclude by infinite descent.
\end{proof}

\begin{remark} There exist $k$-free Kempner sets $\mathcal{K}(S,b)$ for which $S$ is not $k$-free mod $b$. One simple example is the $3$-free set $\mathcal{K}(\{0,2,5\},7)$.  Examples like this rely on gaps in the digit set $S$ (to avoid `carrying') and do not seem to produce large harmonic sums.
\end{remark}

\section{Harmonic Sums of (Shifted) Kempner Sets}

We now turn our attention to the harmonic sums of Kempner sets. In general, let $\mathcal{H}(S)$ denote the harmonic sum of the integer set $S$.  The following result shows that any lower bound for $M_k$ can be approximated by the harmonic sum of a (shifted) $k$-free Kempner set,

\begin{theorem} \label{thm:kempner_approximation}
Let $S$ be $k$-free, with a convergent harmonic series.  Given $\epsilon >0$, there exists a $k$-free Kempner set $\mathcal{K}$ such that $\mathcal{H}(\mathcal{K}+1) > \mathcal{H}(S)-\epsilon$.
\end{theorem}

\begin{proof} Fix $\epsilon > 0$ and choose $M$ such that $\mathcal{H}(S \cap [1,M]) > \mathcal{H}(S)-\epsilon$. Fix an integer $b > \max(k,2M)$, so that $S \cap [1,M]$ and hence $(S \cap [1,M])-\min(S)$ are $k$-free mod $b$ by Proposition~\ref{prop:kfree_extension}.  Then Proposition~\ref{prop:kfree_mod_p_to_kempner} implies that both $\mathcal{K}=\mathcal{K}(S \cap [1,M]-\min(S),b)$ and the shifted set $\mathcal{K}+1$ are $k$-free.

Yet $\mathcal{K}+1$ contains a copy of $S\cap [1,M]$, shifted by $1-\min(S)\leq 0$ (i.e. held constant or decreased), hence $\mathcal{H}(\mathcal{K}+1)\geq \mathcal{H}(S\cap [1,M])>\mathcal{H}(S)-\epsilon$.
\end{proof}

One of the reasons to study Kempner sets is that machinery exists due to \cite{BaillieSchmelzer08} to compute harmonic sums of Kempner sets with great precision.  There is one small difficulty, in that Kempner sets include $0$.  Rather than exclude $0$, we opt to increase our sets termwise by $1$. This shift affects harmonic sum in a way that can be addressed with the following lemma.

\begin{lemma} \label{shift_harmonic}
Let $S$ be a set of positive integers and let $H_n$ denote the $n$th harmonic number.  Then
\[\sum_{s \in S} \frac{1}{s+n} = \sum_{s \in S} \frac{1}{s} + \sum_{s \notin S} \frac{n}{s(s+n)} - H_n.\]
\end{lemma}

\begin{proof} Since $\sum_{s \in S} 1/s - \sum_{s \in S} 1/(s+n) = \sum_{s \in S} n/(s^2+ns)$, it suffices after rearranging to show that $H_n = \sum_{m=1}^\infty n/(m^2+nm)$.  To prove this last fact, we write the series on $m$ as a telescoping sum.
\end{proof}

\section{Implementation}

The Baillie--Schmelzer Algorithm described in~\cite{BaillieSchmelzer08} has been fully implemented in the Mathematica language and is freely available from the Wolfram Library Archive~\cite{BaillieSchmelzerWolframLibrary}.  This is useful for fine-tuning pruned results but inefficient for larger searches because the Baillie--Schmelzer Algorithm is somewhat time-intensive.

As a complement to the Mathematica implementation of the Baillie--Schmelzer algorithm, the author wrote a family of search programs in C++.  The core algorithm is a branch-and-bound depth-first search through the subsets of $[0,b-2]$ which are $k$-free mod $b$. More specifically, states are stored as pairs $(S,T)$, in which $S$ is a $k$-free set mod $b$ and $T$ is the set of possible extensions to $S$:
\[T=\{t \in [0,b-2] : t > \max(S) \text{ and } S \cup \{t\} \text{ is } k\text{-free mod } b\}.\]
An upper bound for the branch rooted at $(S,T)$ is then $\mathcal{H}(\mathcal{K}(S \cup T,b)+1)$.

To efficiently score a candidate set $S$, one may estimate the full harmonic sum $\mathcal{H}(\mathcal{K}(S,b)+1)$ as a function of the single-digit sum $h_1 := \sum_{s \in S} 1/(s+1)$. If one approximates the contribution to $\mathcal{H}(\mathcal{K}(S,b)+1)$ from integers in $[b^n, b^{n+1}]$ as $h_1 \cdot (\#S/b)^n$, one obtains the heuristic harmonic sum
\[
  \mathcal{H}(\mathcal{K}(S,b)+1) \approx \frac{h_1}{1-\#S/b}.
\]
While this approximation is too coarse in practice, the two-digit variant based on $h_2 := \sum_{s,t \in S} 1/(sb+t+1)$ appears tolerable. In practice, we use a refined approximation of the form
\begin{align}\label{eq:kempner_sum_approximation}
  \mathcal{H}(\mathcal{K}(S,b)+1) \approx
  \frac{(1+\alpha \cdot (\# S /b)^{\beta}) h_2 - \gamma \cdot (\# S /b)^2}{1-(\# S /b)^2},
\end{align}
in which $\alpha, \beta, \gamma$ are best-fit parameters (depending on $b$) computed in advance by sampling $10000$ choices for $S \subset [0,b-1]$, drawn from a distribution in which different values of $\# S$ are equiprobable. To give a sense of scale, we have $\alpha = 0.0852$, $\beta = 2.1534$, and $\gamma =  1.4085$ when $b=100$.

Branches $(S,T)$ whose approximate upper bounds (via~\eqref{eq:kempner_sum_approximation}) lie below a threshold are pruned, and surviving sets are recorded for further processing using the full Baillie--Schmelzer algorithm (and Lemma~\ref{shift_harmonic}) in Mathematica.
Code for searching and post-processing these sets, as well as log files from recorded searches, is available as a GitHub repository~\cite{CodeRepo}.

\section{\texorpdfstring{$3$}{3}-Free Kempner Sets of Large Harmonic Sum}

The author performed a branch-and-bound search over $3$-free sets mod $b$ for each $b \leq 158$ in an attempt to find Kempner sets which outperformed the greedy set $G_3 = \mathcal{K}(\{0,1\},3) + 1$. This search was unsuccessful, indicating that $G_3$ is an influential local maximum and that larger search spaces will be needed to improve lower bounds on $M_3$.

To extend these results to larger $b$, at least conditionally, the author ran searches with additional heuristic pruning. Two methods were considered:
\begin{enumerate}
  \item[P1.] One asserts an upper bound on $\# S$ and reruns any search in which the upper bound was realized.
  \item[P2.] One restricts to $3$-free sets which deviate from a greedy construction for $3$-free sets mod $b$ a bounded number of times.
\end{enumerate}
In addition to the unconditional search over $b \leq 158$, the author ran searches under (P1) for $b \in [159, 316]$ and under (P1) and (P2) for $b \in [317,400]$. These searches, which ran for a total of $8679$ core-hours (at 2.3 GHz), were unsuccessful in improving $G_3$. The largest harmonic sums found associated to each base $b \leq 400$ are depicted in Figure~\ref{fig:harmonic3}.

\begin{figure}[ht]
  \includegraphics[width=\textwidth]{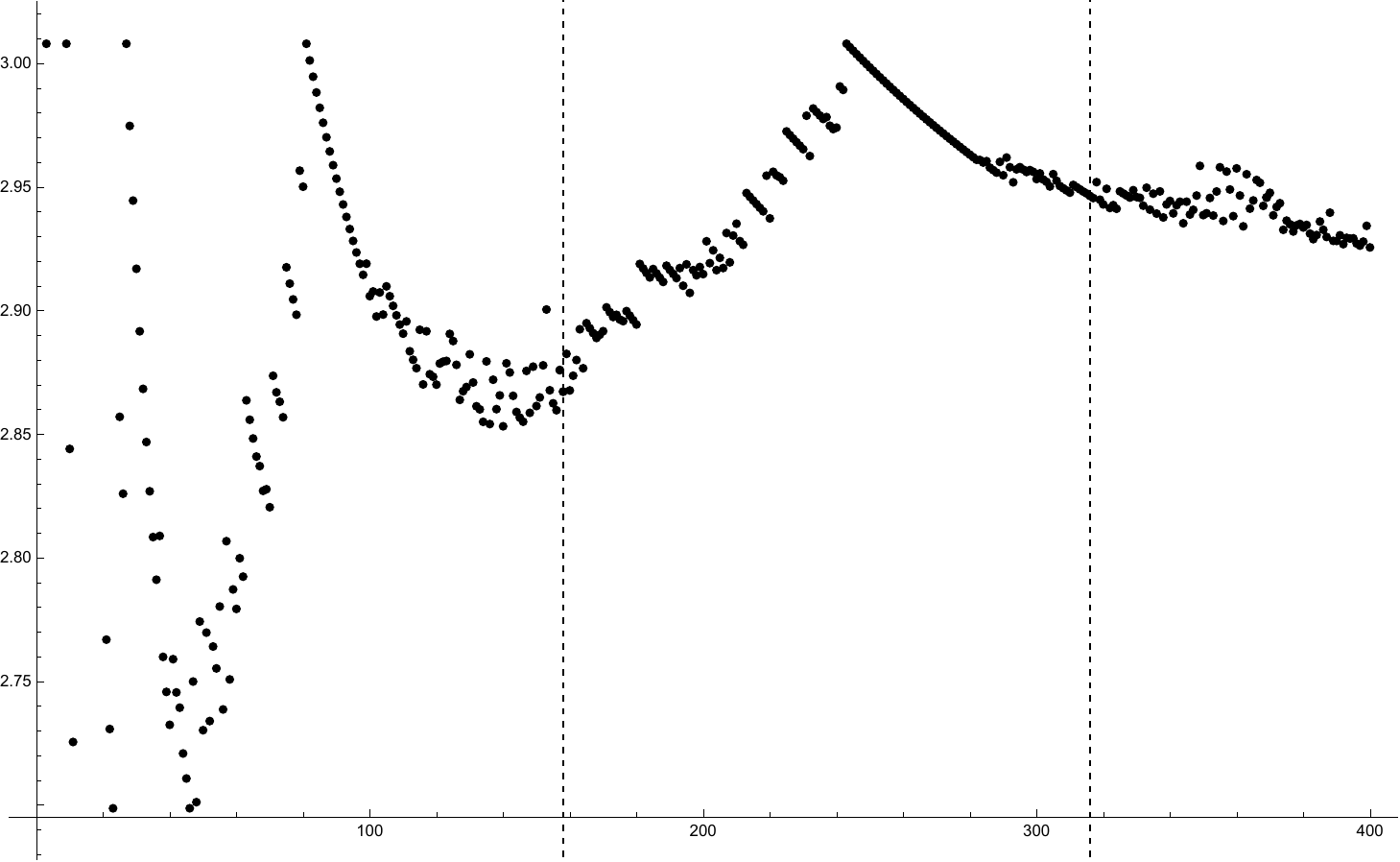}
  \caption{\footnotesize Largest harmonic sums found in bases $b \leq 400$. Maxima at $b \in \{3,9,27,81,243\}$ correspond to $G_3$. Results for $b \geq 370$ are likely non-optimal.}
  \label{fig:harmonic3}
\end{figure}

\begin{remark}
The behavior described here for $k=3$ is typical of prime $k$: the greedy set $G_k=\mathcal{K}([0,k-2],k)+1$ dominates early results and we find little else of interest. When $k$ is composite, $G_k$ is non-Kempner and early results are dominated by $G_p$, where $p$ is the largest prime less than $k$. This inefficiency is most pronounced when $k$ is one less than a prime.
\end{remark}

\section{Four-Free Kempner Sets of Large Harmonic Sum}

We next consider $k$-free sets in the case $k=4$. Heuristics from~\cite{GerverRamsey79} suggest the lower bound $M_4 > 4.3$, as derived from the greedy set $G_4$.  (No algorithm is known for computing $\mathcal{H}(G_4)$ with precision; a lower bound using the first $100000$ terms of $G_4$ gives $\mathcal{H}(G_4) > 4.25146$.) 

A series of branch-and-bound searches over all $4$-free sets mod $b \leq 88$ yields $5$ shifted Kempner sets with harmonic sum above $4.35$. Our best result employs a $21$-term $4$-free set mod $55$ and shows $M_4 \geq 4.43975$, which improves the heuristic record set by~\cite{GerverRamsey79}.

Under the pruning heuristic (P1), we extend these results to $b \leq 118$, finding two additional sets with harmonic sum over $4.35$. Finally, under (P1) and (P2), we consider $b \leq 200$ and produce $390$ additional sets above the threshold $4.35$. Of these, $290$ have $b=121$, mostly as deleterious modifications of $\mathcal{K}(\{0,1,2,4,5,7\}, 11)$. In aggregate, these searches ran for $5603$ core-hours (at $2.3$ GHz). The ten sets of largest harmonic sum (excluding results from $b=121$) found are compiled in Table~\ref{tab:4_free}.

\begin{center}
\begin{longtable}{c|c|l}
\caption*{\textsc{\tablename}\ \thetable. $4$-free Kempner Sets with Large Harmonic Sum}\\
$\mathcal{H}(\mathcal{K}+1)$ & $b$ & $S$ \\
\hline
\endfirsthead
\multicolumn{3}{c}%
{\textsc{\tablename}\ \thetable. \textit{Continued from previous page}}\vspace{1 mm} \\
$\mathcal{H}(\mathcal{K}+1)$ & $b$ & $S$ \\\hline
\endhead
\hline \multicolumn{3}{r}{\textit{Continued on next page}} \\
\endfoot
\hline
\endlastfoot
4.43975 & 55 & \footnotesize \{0,1,2,4,5,9,10,11,14,16,17,18,21,24,30,37,39,41,42,45,47\} \\
4.42175 & 11 & \footnotesize \{0,1,2,4,5,7\} \\
4.41989 & 22 & \footnotesize \{0,1,2,4,5,7,8,9,14,17\} \\
4.39620 & 191 & \footnotesize\{0,1,2,4,5,7,8,9,14,16,17,18,26,30,31,32,36,37,39,40,42,50,\\*
& & \footnotesize \qquad 51,55,56,58,59,62,64,67,69,70,77,81,83,87,91,94,102,\\*
& & \footnotesize \qquad 109,110,111,113,117,119,120,122,123,125,127\} \\
4.37859 & 177 & \footnotesize\{0,1,2,4,5,7,8,9,14,16,17,18,22,29,30,31,34,35,37,39,42,45,\\*
& & \footnotesize \qquad 47,49,57,58,61,63,65,66,70,71,72,78,80,81,82,89,96,\\*
& & \footnotesize \qquad 100,102,108,110,116,136,149\} \\
4.37699 & 193 & \footnotesize\{0,1,2,4,5,7,8,9,14,16,17,18,22,29,30,31,34,35,37,39,42,45,\\*
& & \footnotesize \qquad 47,49,57,58,60,61,64,65,66,70,71,72,74,92,96,100,102,\\*
& & \footnotesize \qquad 106,110,113,116,117,118,122,124,125,157\} \\
4.37665 & 157 & \footnotesize\{0,1,2,4,5,7,8,9,14,16,17,18,22,28,29,30,32,35,36,37,39,45,\\*
& & \footnotesize \qquad 57,59,61,62,67,68,69,71,75,76,78,80,84,95,104,108,115,\\*
& & \footnotesize \qquad 119,137,142,146\} \\
4.37583 & 97 & \footnotesize\{0,1,2,4,5,7,8,17,18,20,21,23,24,25,30,32,37,45,48,54,56,58,\\*
& & \footnotesize \qquad 59,61,63,64,66,68,74,77,85,90,92\} \\
4.37486 & 193 & \footnotesize\{0,1,2,4,5,7,8,9,14,16,17,18,22,29,30,31,34,35,37,39,42,45,\\*
& & \footnotesize \qquad 47,49,57,58,60,61,64,65,66,70,71,72,74,92,96,100,102,\\*
& & \footnotesize \qquad 106,113,116,117,118,122,124,125,128,157\} \\
4.37406 & 105 & \footnotesize\{0,1,2,4,5,7,8,9,15,16,18,19,20,25,26,28,29,31,32,33,36,45,\\*
& & \footnotesize \qquad 50,51,59,61,63,68,70,72,79\}
\label{tab:4_free}
\end{longtable}
\end{center}

\section{Ten-Free Kempner Sets of Large Harmonic Sum}

The author ran searches for $k$-free sets with $k=6$ and $k=10$ as well. For $k=6$, we ran an unconstrained search for $b \leq 60$, extended to $b \leq 74$ under (P1) and $b \leq 125$ under (P1) and (P2). In total, these searches ran for $4199$ core-hours (at $2.3$ GHz) but were unable to improve the lower bound $M_6 \geq \mathcal{H}(\{1\} \cup (G_5 + 1)) = 7.94433$.

For $k=10$, we ran an unconstrained search for $b \leq 58$ as well as a search for $b \leq 100$ under (P1) and (P2). This search, which ran for $1607$ core-hours (at $2.3$ GHz) found $66$ sets with harmonic sum at least $13.5$. Our best result builds on the $55$-term $10$-free set 
\begin{align*}
  S = \{&0,1,2,3,4,5,6,7,8,10,11,12,13,14,15,17,18,19,20,21,\\
  &22,24,25,26,27,28,29,31,32,33,34,35,36,38,39,40,42,\\
  &43,45,46,47,48,49,50,52,53,55,56,60,61,62,68,69,71,73\}
\end{align*}
with base $b=77$ and establishes $M_{10} \geq 14.056$, which dramatically improves the simple lower bound $\mathcal{H}(\{1,2,3\} \cup (G_7 + 1)) = 13.5905$.

\begin{remark}
Fix an integer $m \geq 1$ and an integer set $A \subset [1,m]$. Since $G_7$ is $7$-free, any $10$-term arithmetic progression in $A \cup (G_7 + m)$ must include at least four elements of $A$. We therefore obtain an efficient test for checking whether such sets are $10$-free. While this method can produce some interesting results, such as
\[
  \mathcal{H}(
    \{1,2,3,4,5,6,7,8,9,11,12,13,14,15,16\} \cup (G_7+17)
  )
  = 13.6962,
\]
it seems unlikely to improve the lower bound $M_{10} \geq 14.056$.
\end{remark}

\vspace{5 mm}
\bibliographystyle{alpha}
\bibliography{compiled_bibliography}

\end{document}